\documentclass[]{amsart}
\usepackage{amsmath}
\usepackage{amsfonts}
\usepackage{amssymb}
\usepackage[all]{xy}
\newtheorem{teo}{Theorem}
\newtheorem{theorem}{Theorem}[section]

\newtheorem{proposition}[theorem]{Proposition}

\def\m{\ensuremath{\mathcal{M}}}

\def\b{\ensuremath{\mathcal{B}}}
\def\c{\ensuremath{\mathcal{C}}}

\def\cat{\ensuremath{\mathbf{Cat}}}
\def\mcat{\ensuremath{\mathbf{MonCat}}}
\def\bmcat{\ensuremath{\mathbf{BrMonCat}}}

\def\set{\ensuremath{\mathbf{Set}}}

\def\diag{\ensuremath{\mathrm{diag}}}

\newcommand{\sset}{\ensuremath{\mathbf{Simpl.Set}}}
\newcommand{\bicat}{\ensuremath{\mathbf{Bicat}}}
\newcommand{\bsset}{\ensuremath{\mathbf{Bisimpl.Set}}}

\begin{document}
\title[{\em Homotopy colimit theorem for braided monoidal categories}]{A homotopy colimit theorem for diagrams of braided monoidal categories}
\author{A. R. Garz\'{o}n}
\author{R. P\'{e}rez}
\thanks{This work has been partially supported by DGI of Spain and FEDER (Project: MTM2007-65431); Consejer\'{i}a de Innovaci\'{o}n
de J. de Andaluc\'{i}a (FQM-168).}
 \address{Departamento de \'Algebra, Facultad de Ciencias, Universidad de Granada}
 \address{ 18071 Granada, Spain.}
\address{agarzon@ugr.es  }
\begin{abstract}
Thomason's Homotopy Colimit Theorem has been extended to bicategories and this extension can be adapted, through the delooping principle, to a corresponding theorem for diagrams of monoidal categories. In this version, we show that the homotopy type of the diagram  can be also  represented by a genuine simplicial set nerve associated with it. This suggests the study of a homotopy colimit theorem, for diagrams $\b$ of braided monoidal categories, by means of a simplicial set {\em nerve of the diagram}. We prove that it is weak homotopy equivalent to the homotopy colimit of the diagram, of simplicial sets, obtained from composing $\b$ with the geometric nerve functor of braided monoidal categories.

\end{abstract}
\maketitle

\section{Introduction}
The Grothendieck construction    $\int_{I}\mathcal{C}$ \cite{Grothendieck1971} on a  diagram of categories ${\mathcal{C}: I^{op}\rightarrow \cat}$  has been recently  extended to the context of 2-categories \cite{Ceg-2011} and, more broadly, of bicategories  \cite{Carrascoa}. These facts allowed then to show extensions, to diagrams of 2-categories \cite[Theorem 4.5]{Ceg-2011} and, to diagrams of bicategories \cite[Theorem 7.3]{Carrascoa}, of the well-known Thomason's Homotopy Colimit Theorem \cite{Thomason1979}. Through the elemental delooping construction \cite{Kapranov1994}, a monoidal category $\mathcal{M}=(\mathcal{M},\otimes,a,\text{I},l,r)$ \cite{Mac1998} can be regarded as a bicategory $\Omega^{^{-1}}\hspace{-5pt} \mathcal{M}$ with only one object, and therefore that extension provides a corresponding theorem for diagrams of  monoidal categories. In this theorem, we remark that the homotopy type of the homotopy colimit of the diagram can be also represented by a genuine simplicial set nerve associated with the diagram. A braided monoidal category $\mathcal{M}=(\mathcal{M}, \otimes,a,\mathrm{I},l,r,\boldsymbol{c})$ \cite{Joyal1993} defines, by double delooping, a one-object, one-arrow tricategory $\Omega^{^{-2}}\hspace{-5pt}\m$ \cite{Berger1999, Kapranov1994}. Although neither Grothendieck construction nor homotopy colimit theorem is known for diagrams of tricategories, the above remark suggests an extension of Thomason's theorem, to diagrams of braided monoidal categories $\b:I^{op}\rightarrow \bmcat$. The composite of $\b$ with the geometric nerve $Ner$ of braided monoidal categories \cite{CarCG-2010} gives  a diagram of simplicial sets, whose homotopy colimit is the matter of study. For it, we use a notion of nerve $Ner_{I}\b$, associated with the diagram $\b$,   that is a particular case of a general notion of nerve of a pseudofunctor of braided monoidal categories introduced in \cite{Cegarra2007a}.
We actually show that $Ner_{I}\b$ is able to represent the homotopy type of the diagram through the existence (Theorem \ref{teor}) of a natural weak homotopy equivalence $\eta : hocolim_{I} Ner \b\rightarrow Ner_{I}\b\,.$

\section{Preliminaries}
We start  fixing notations and terminology and reviewing necessary results from the background of (bi)simplicial sets  used throughout the paper. We employ the standard symbolism and nomenclature to be found in texts on simplicial homotopy theory (see \cite{May1967, Goerss1999}).

Hereafter, we shall regard each ordered set $[n]=\{0,1,\cdots n\}$\ as the category with exactly one arrow $j\rightarrow i$\ if $i\leq j$. Then, a non-decreasing map $[m]\rightarrow[n]$\ is a functor, so that we can see $\Delta$, the simplicial category of finite ordinal numbers, as a full subcategory of $\cat$, the category of small categories. $\sset$ will denote the category of simplicial sets, that is, functors $S: \Delta^{op}\rightarrow \set$.  A \textit{weak homotopy equivalence} of simplicial sets is a simplicial map whose geometric realization is a homotopy equivalence.

A bisimplicial set is a functor $S: \Delta^{op}\times \Delta^{op}\rightarrow \set$. This amounts to a family of sets  $\{S_{p,q};\ p,q\geq 0 \} $\ together with horizontal and vertical face and degeneracy operators

$$\xymatrix{S_{p+1,q}&\ar[l]_-{\textstyle s_i^h}S_{p,q}\ar[r]^-{\textstyle d_i^h}&S_{p-1,q}, \hspace{0.6cm}
S_{p,q+1}&\ar[l]_-{\textstyle s_j^v} S_{p,q}\ar[r]^-{\textstyle d_j^v}&S_{p,q-1},}
$$
with $0\!\leq i\!\leq p$ and $0\!\leq j\!\leq q$ respectively, such that, for all $p$ and $q$, both
$S_{p,\ast}$ and $S_{\ast,q}$ are simplicial sets and the horizontal operators commute with the
vertical ones. $\bsset$ will denote the category of bisimplicial sets.

We shall  use the bar construction on a bisimplicial set $\overline{W}S$, also called its {\em codiagonal} or
{\em total complex}. Let us recall that the functor
$$
\overline{W}\!\!: \bsset \to \sset
$$
  can be described, for any given bisimplicial set $S$, as follows \cite[\S III]{Artin1966}: the set of $p$-simplices of $\overline{W}S$ is
$$
\Big\{(t_{0,p} \dots,t_{p,0})\in \prod_{m=0}^{p}S_{m,p-m}~|~d^v_0t_{m,p-m}=d^h_{m+1}t_{m+1,p-m-1},\, 0\!\leq m\!< p\Big\}$$
and, for $0\!\leq i\!\leq p$, the faces and degeneracies of a $p$-simplex are given by
$$
\begin{array}{l}d_i(t_{0,p}
\dots,t_{p,0})=(d_i^vt_{0,p},\dots,d_i^vt_{i-1,p-i+1},d_i^ht_{i+1,p-i-1},\dots,d_i^ht_{p,0}),
\\[-4pt]~\\
s_i(t_{0,p} \dots,t_{p,0})=(s_i^vt_{0,p},\dots,s_0^vt_{i,p-i},s_i^ht_{i,p-i},\dots,s_i^ht_{p,0}).
\end{array}
$$
On the other
hand, by composing with the diagonal functor $\diag\!\!:\Delta^{\!{^\mathrm{op}}}\rightarrow\Delta^{\!{^\mathrm{op}}}\times\Delta^{\!{^\mathrm{op}}}$,
the bisimplicial set $S$ also provides another simplicial set $\diag \,S\!\!:[n]\mapsto S_{n,n}$, whose
face and degeneracy operators are given in terms of those of $S$ by the formulas $d_i=d_i^hd_i^v$
and $s_i=s_i^hs_i^v$, respectively.

For any bisimplicial set $S$, there is a natural weak homotopy equivalence \cite{Cegarra2005a,Cegarra2007b}
\begin{equation}\label{Phi} \Phi\!\!:\diag\,S\to \overline{W}S,
\end{equation}
which carries a $p$-simplex $t_{p,p}\in \diag \,S$
to
$$
\Phi t_{p,p}=\Big((d_1^h)^pt_{p,p},(d_2^h)^{p-1}d_0^vt_{p,p},\dots, (d_{m+1}^h)^{p-m}(d_0^v)^mt_{p,p},\dots,
(d_0^v)^pt_{p,p}\Big).
$$

Next subsection is devoted to recall the extension to bicategories, shown in \cite{Carrascoa}, of the Thomason's homotopy colimit theorem.

\subsection{Thomason's Homotopy Colimit Theorem for bicategories}
 Recall that the Grothendieck nerve functor $Ner : \cat \rightarrow \sset$  associates, with every small category $\mathcal{C}$, the simplicial set $Ner\, \mathcal{C}$
whose $n$-simplices are all functors ${F :[n]\rightarrow \mathcal{C}}$ or, equivalently, tuples of arrows in $\mathcal{C},\ F=(Fj \xrightarrow{F_{i,j}} Fi)_{0 \leq i\leq j\leq n}$, such that $F_{i,j}F_{j,k}=F_{i,k}$\ for $i\leq j \leq k$\ and $F_{i,i}=1_{Fi}$.
 For any given small category $I$ and  any diagram of categories  $\mathcal{C}: I^{op}\rightarrow \cat$, the composite $Ner \mathcal{C}: I^{op}\rightarrow \sset$ is just a diagram of simplicial sets. The homotopy colimit construction by Bousfield and Kan \cite{Bousfield1972} of this diagram is the simplicial set
$$\begin{aligned}
hocolim_{I}Ner\mathcal{C}: &\ \Delta^{op} \rightarrow \set\\
&\left[ n \right] \hspace*{0,25cm}  \mapsto  \underset{\tiny{\left[ n \right] \xrightarrow{\sigma} I}}{\coprod} Func([n],\mathcal{C}_{\sigma 0})\,.
\end{aligned}$$
Considering Grothendieck construction  $\int_{I}\mathcal{C}$ of the diagram $\mathcal{C}: I^{op}\rightarrow \cat$,  Thomason's homotopy colimit theorem states the following:

\begin{teo}(\cite[Theorem 1.2]{Thomason1979})
For any diagram of categories $\mathcal{C}: I^{op}\rightarrow \cat$ there is a natural weak homotopy equivalence
$$
\eta:  hocolim_{I}Ner\mathcal{C}\rightarrow Ner \int_{I}\mathcal{C}\,.
$$
\end{teo}
In this way, the classifying space of the category  $\int_{I}\mathcal{C}$ (that is, the geometric realization of its nerve) can be thought of as a homotopy colimit of the classifying spaces of the categories $\mathcal{C}_i$ that arise from the initial input data $i\rightarrow \mathcal{C}_i$ given by the diagram of categories $\mathcal{C}$.

Below we quickly review the extension to bicategories, given  in \cite{Carrascoa}, of the above theorem. For the background concerning bicategories we refer to \cite{Benabou1967, Gordon1995, Street1996}. $\bicat$ will denote the category of bicategories and homomorphisms between them (i.e., lax functors  where the structure constraints are invertible). For any diagram of bicategories  $\mathcal{C}: I^{op}\rightarrow \bicat$ there is  \cite{Carrascoa} a bicategorical Grothendieck construction $\int_{I}\mathcal{C}$ that suitably assembles all bicategories $\mathcal{C}_i$, $i\in \text{Ob}I$.
Also, for any  small bicategory $\mathcal{C}$, we recall that the \textit{geometric nerve} of  $\mathcal{C}$ is the simplicial set,
$$\begin{aligned}
\Delta \mathcal{C}:& \ \Delta^{op}\rightarrow Set\\
& \left[ n \right] \hspace*{0,25cm} \mapsto LaxFunc([n], \mathcal{C})
\end{aligned}$$
whose  $n$-simplices are all lax functors $F:[n]\rightarrow \mathcal{C}$. The \textit{unitary geometric nerve} of $\mathcal{C}$,  $\Delta^{u}\mathcal{C}$, is defined considering only normal (or unitary, i.e., when the unit constraints are all identities) lax functors $F:[n]\rightarrow \mathcal{C}$.

 Then, Thomason's theorem admits the following generalization to diagrams of bicategories:
\begin{teo}\label{t}\cite[Theorem 7.3]{Carrascoa}
For any diagram of bicategories $\mathcal{C}: I^{op}\rightarrow \bicat$, there exists a natural weak homotopy equivalence of simplicial sets
\begin{eqnarray}\label{eta}
\eta: hocolim_{I}\Delta \mathcal{C}\rightarrow \Delta \int_{I}\mathcal{C}
\end{eqnarray}
where $\Delta \mathcal{C}: I^{op}\rightarrow \sset$\ is the diagram of simplicial sets obtained by composing  $\mathcal{C}$ with the geometric nerve functor $\Delta: \bicat \rightarrow \sset$.
\end{teo}
Ten different (but homotopically equivalent) ways of defining the classifying space of any bicategory have been shown in \cite[Theorem 6.1]{Carrascoa}. The above theorem  stablishes then how the classifying space of the bicategorical Grothendieck construction $\int_{I}\mathcal{C}$ can be thought of as a homotopy colimit of the classifying spaces of the bicategories $\mathcal{C}_i$ that arise from the initial input data $i\mapsto \mathcal{C}_i$ given by the diagram of bicategories $\mathcal{C}$.

 In the next subsection we show how Theorem \ref{t} provides a corresponding one for monoidal categories in which, besides, the homotopy colimit can be represented by a genuine simplicial set nerve of the diagram.

\subsection{Homotopy colimit theorem for monoidal categories}

In what follows $\mcat$ will denote the category of monoidal categories and monoidal functors between them.
Note that every monoidal category $\mathcal{M}=(\mathcal{M},\otimes,a,\text{I},l,r)$\ can be regarded as a bicategory $\Omega^{^{-1}}\hspace{-5pt} \mathcal{M}$\ with only one object $\ast$, whose morphisms are the objects of $\mathcal{M}$\ and whose deformations are the morphisms of  $\mathcal{M}$.
The horizontal composition is given by the tensor functor  $\otimes : \mathcal{M}\times \mathcal{M}\rightarrow \mathcal{M}$, the identity at the object $\ast$ is the unit object I of $\mathcal{M}$ and the associativity, left and right unit constraints for  $\Omega^{^{-1}}\hspace{-5pt} \mathcal{M}$ are just those of the monoidal category. This observation, due to J. Benabou \cite{Benabou1967}, that monoidal categories are essentially the same as bicategories with just one object, is known as the {\em delooping principle}, and the bicategory $\Omega^{^{-1}}\hspace{-5pt}\m$ is called the {\em delooping of the monoidal category} \cite[2.10]{Kapranov1994}. Thus we have the delooping embedding $\Omega^{^{-1}}:\mcat\rightarrow \bicat$.

If $\mathcal{M}: I^{op}\rightarrow \mcat$, $(j\overset{a}\to i)\mapsto   (\m_i\overset{a^*}\to \m_j)$, is a diagram of monoidal categories, it follows from Theorem \ref{t} that the homotopy type of $\m$ is modeled by the bicategory $\xymatrix{\int_I\!\Omega^{^{-1}}\hspace{-5pt}\m,}$  after the existence, according to (\ref{eta}), of a weak homotopy equivalence
\begin{equation}\label{whe}\eta: hocolim_{I}\Delta \Omega^{^{-1}}\hspace{-5pt}\m\rightarrow \Delta \int_{I}\Omega^{^{-1}}\hspace{-5pt}\m\,. \end{equation}

 But, as we shall see in detail below, the homotopy type of the diagram $\m$ can be also represented by a simplicial set associated to it, namely $ \mbox{Ner}_I\m$, the nerve of the diagram defined in \cite{Cegarra2007a}.

 Recall that a {\em \text{2}-cocycle} of $I$ with coefficients in $\m$ is  a
system of data $(Y,f)$ consisting of:

- For each arrow $j\overset{a}\to i$ in $I$, an object $Y_a\in
\m_j$.

- For each pair of composable arrows in  $I$, $k\overset{b}\to
j\overset{a}\rightarrow i$, a morphism in $\m_k$
$$\xymatrix{b^*Y_a\otimes Y_b\ar[r]^-{\textstyle f_{a,b}}&Y_{ab},}$$
such that  $Y_{1_j}=\mathrm{I}$ (the unit object of $\m_j$), the morphisms $f_{1,a}\!\!:a^*\mathrm{I}\otimes Y_a\to Y_a$ and $f_{a,1}\!\!:Y_a\otimes \mathrm{I} \to Y_a$ are the canonical isomorphisms given by the unit constrains of the monoidal category $\m_j$ and the monoidal functor $a^*$, and for any three composable triplet, $\ell\overset{c}\to k
\overset{b}\to j\overset{a}\to i$, of morphisms in $I$, the  coherence condition given by the commutativity of the following diagram in $\m_\ell$
\begin{equation}\label{coh}\xymatrix@C=50pt@R=25pt{(c^*b^*Y_a\otimes c^*Y_b)\otimes Y_c\ar[r]^{\cong}&c^*(b^*Y_a\otimes Y_b)\otimes Y_c\ar[r]^-{\textstyle c^*\!f_{a,b}\otimes 1}&c^*Y_{ab}\otimes Y_c \ar[d]^{\textstyle f_{ab,c}}\\c^*b^*Y_a\otimes(c^*Y_b\otimes Y_c)\ar[u]^{\cong}\ar[r]^-{\textstyle 1\otimes f_{b,c}}&c^*b^*Y_a\otimes Y_{bc}\ar[r]^-{\textstyle f_{a,bc}}& Y_{abc}}
\end{equation}
(where the unnamed isomorphisms are canonical) must hold.

Then,  $\mbox{Ner}_I\m$, the nerve of the diagram,  is defined as the simplicial set
\begin{equation}\label{4.6} \mbox{Ner}_I\m\!\!:\ [n]\mapsto \bigsqcup_{G\!:[n]\to I} Z^2\big([n],\m\,G\big)\,,
\end{equation}
where $G\!\!:[n]\to I$ is any functor and $Z^2\big([n], \m\,G\big)$ is the
set of $2$-cocycles of $[n]$ in the composite functor
$[n]\overset{G}\longrightarrow I\overset{\m}\longrightarrow \mcat$. Now we can prove the following:
\begin{proposition}
For any diagram of monoidal categories $\mathcal{M}: I^{op}\rightarrow \mcat$, there is a natural isomorphism of simplicial sets
$$ \mbox{Ner}_I\m \cong \Delta^{u} \int_{I}\Omega^{^{-1}}\hspace{-5pt}\m\,.$$
\end{proposition}
 \begin{proof}
 The isomorphism will be clear after describing the simplices of both simplicial sets in each dimension.

 First we observe that, according to the general construction of the bicategorical Grothendieck construction \cite{Carrascoa}, the bicategory $\int_{I}\Omega^{^{-1}}\hspace{-5pt}\m$ has, as objects, pairs $(\ast_i,i)$ where $i\in \text{Ob}I$ and $\ast_i$ is the unique object of $\Omega^{^{-1}}\hspace{-5pt}\m_i$. Thus $\text{Ob}\int_{I}\Omega^{^{-1}}\hspace{-5pt}\m\cong \text{Ob}I$. A morphism from $(\ast_j,j)$ to $(\ast_i,i)$ is a pair $(X,a)$, where $a:j\rightarrow i$ is a morphism in $I$ and $X\in \text{Ob}\m_j$, and a 2-cell from $(X,a)$ to $(Y,a)$ is just a morphism $\alpha:X\rightarrow Y$ in $\m_j$. The horizontal composition functor is given by:\\
$$\xymatrix@C=0,16cm{
(\ast_k,k)\ar@/^5mm/[rr]^{(Z,b)}\ar@/_5mm/[rr]_{(T,b)}& \Downarrow \beta & (\ast_j,j)\ar@/^5mm/[rr]^{(X,a)} \ar@/_5mm/[rr]_{(Y,\sigma)}& \Downarrow \alpha & (\ast_i,i) \ar@{|->}[rr]^{\circ}&\empty& (\ast_k,k)\ar@/^5mm/[rr]^{(b^*X\otimes Z,a\,b)}\ar@/_5mm/[rr]_{(b^*Y\otimes T,a\,b)}& \Downarrow b^*\alpha \otimes \beta & (\ast_i,i)
}$$ where $b^*:\Omega^{^{-1}}\hspace{-5pt}\m_j\rightarrow \Omega^{^{-1}}\hspace{-5pt}\m_k$ is the induced homomorphism by $b:k\rightarrow j$.\\

  Then, the unitary geometric nerve of the bicategory $\int_{I}\Omega^{^{-1}}\hspace{-5pt}\m$\ is the simplicial set whose $n$-simplices, the normal lax functors $[n]\rightarrow \int_{I}\Omega^{^{-1}}\hspace{-5pt}\m$,
are described as follows ( see \cite[Lemma 4.2]{Carrascoa}):\\
- In dimension zero, $(\Delta^u \int_{I}\Omega^{^{-1}}\hspace{-5pt}\m)_0=\text{Ob}\int_{I}\Omega^{^{-1}}\hspace{-5pt}\m\cong \text{Ob}I$

\noindent - The 1-simplices $F:[1]\rightarrow \int_{I}\Omega^{^{-1}}\hspace{-5pt}\m$ of $\Delta^u \int_{I}\Omega^{^{-1}}\hspace{-5pt}\m$ are the morphisms in $\int_{I}\Omega^{^{-1}}\hspace{-5pt}\m$ from $F1=(\ast_{F1},F1)$ to $F0=(\ast_{F0},F0)$, that is,  pairs  $(X_{0,1},F_{0,1})$
where $X_{0,1}:\ast_{F1} \rightarrow \ast_{F1}$ is a morphism in $\Omega^{^{-1}} \hspace{-5pt}\m_{F1}$ (i.e., an object of $\m_{F1}$)
and  $F_{0,1}:F1\rightarrow F0$\ is a morphism in $I$.

\noindent - The 2-simplices $F:[2]\rightarrow \int_{I}\Omega^{^{-1}}\hspace{-5pt}\m$ are triangles
$$\xymatrix{
\empty & (\ast_{F1},F1) \ar[dl]_{(X_{0,1},F_{0,1})} \ar@2[d]^{\widehat{F}_{0,1,2}}& \empty\\
(\ast_{F0},F0) & \empty & (\ast_{F2},F2) \ar[ll]^{(X_{0,2},F_{0,2})}\ar[ul]_{(X_{1,2},F_{1,2})}}$$
where $\widehat{F}_{0,1,2}: (X_{0,1},F_{0,1})\circ (X_{1,2},F_{1,2})\Rightarrow (X_{0,2},F_{0,2})$\ is a deformation in $\int_{I}\Omega^{^{-1}}\hspace{-5pt}\m$. Then, $\widehat{F}_{0,1,2}: (F_{1,2}^*X_{0,1}\otimes X_{1,2},F_{0,1}F_{1,2})\Rightarrow (X_{0,2},F_{0,2})$ and therefore
 $F_{0,1}F_{1,2}=F_{0,2}$ and $\widehat{F}_{0,1,2}$ is just  a  deformation $F_{0,1,2}:F_{1,2}^*X_{0,1}\otimes X_{1,2}\rightarrow X_{0,2}$\  in $\Omega^{^{-1}}\hspace{-5pt}\m_{F2}$, that is,  a morphism in $\m_{F2}$ from $F_{1,2}^*X_{0,1}\otimes X_{1,2}$ to $ X_{0,2}$.

 Then, a 2-simplex of $\Delta^u \int_{I}\Omega^{^{-1}}\hspace{-5pt}\m$\ is a tuple $(F0,F1,F2,F_{0,1},F_{0,2},F_{1,2},F_{0,1,2})$
 where $Fi$, $i=0,1,2$, are objects of $I$, $F_{i,j}:Fj\rightarrow Fi$, $0\leq i< j \leq 2$, are morphisms in $I$ such that $F_{0,1}F_{1,2}=F_{0,2}$ and $F_{0,1,2}:F_{1,2}^*X_{0,1}\otimes X_{1,2}\rightarrow X_{0,2}$ is a morphism in $\m_{F2}$ with $X_{0,1}\in \text{Ob}\,\m_{F1}$ and $X_{1,2},X_{0,2}\in \text{Ob}\,\m_{F2}$.

- If $n\geq 3$,  a $n$-simplex $F:[n]\rightarrow \int_{I}\Omega^{^{-1}}\hspace{-5pt}\m$  is determined by
objects $Fi$\ of $I$, $ 0 \leq i \leq n$; for any $0\leq i<j\leq n$, by objects $X_{i,j}\in \text{Ob}\,\m_{Fj}$ and morphisms in $I$, $F_{i,j}: Fj \rightarrow Fi$, such that, for any $0\leq i<j<k\leq n$, $F_{i,j}\,F_{j,k}=F_{i,k}$; and, for any $0\leq i<j<k\leq n$, by morphisms in $\m_{Fk}$, $F_{i,j,k}: F_{j,k}^* X_{i,j} \otimes X_{j,k} \rightarrow X_{i,k}$ such that, for any $0\leq i< j< k < \ell \leq n$, the following diagram in $\m_{F\ell}$, where `can' denotes a composite of canonical isomorphisms, is commutative:

$$\xymatrix@C=47pt{
F_{k,\ell}^*F_{j,k}^* X_{i,j}\otimes F_{k,\ell}^*X_{j,k}\otimes X_{k,\ell}\ar[d]_{\text{can}}\ar[rr]^{1\otimes F_{j,k,\ell}}& \empty & F_{j,\ell}^* X_{i,j}\otimes X_{j,\ell} \ar[d]^{F_{i,j,\ell}}\\
F_{k,\ell}^*(F_{j,k}^*X_{i,j}\otimes X_{j,k})\otimes X_{k,\ell} \ar[r]^-{F^{*}_{k,\ell}(F_{i,j,k})\otimes 1} & F_{k,\ell}^*X_{i,k}\otimes X_{k,\ell} \ar[r]^{F_{i,k,\ell}} & X_{i,\ell}
}$$

The whole data giving the normal lax functor $F:[n]\rightarrow \int_{I}\Omega^{^{-1}}\hspace{-5pt}\m$ are obtained by putting $X_{i,i}=I$, $0\leq i\leq n$, where $I$ is the unit object of $\m_{Fi}$, and $F_{i,i,j}:F_{i,j}^*I\otimes X_{i,j}\rightarrow X_{i,j}$ and $F_{i,j,j}:X_{i,j}\otimes I\rightarrow X_{i,j}$, $0\leq i<j\leq n$,  given by canonical (left, right, unit) constraints of the monoidal category $\m_{Fj}$.

As for the n-simplices of $Ner_{I}\m$, $$\bigsqcup_{G:[n]\to I} Z^2\big([n],\m\,G\big)\,,$$ they can be described as follows:\\
\noindent - A 0-simplex consist of a functor $G:[0]\rightarrow I$ and a 2-cocycle of $[0]$ with coefficients in $\m G$. Thus it is determined by the objects $G0$ of $I$ and therefore $$(Ner_{I}\m)_0\cong  \text{Ob}(I)\cong (\Delta^u \int_{I}\Omega^{^{-1}}\hspace{-5pt}\m)_0\,.$$

\noindent - A 1-simplex is given by a functor $G:[1]\rightarrow I$ and a 2-cocycle of $[1]$ with coefficients in $\m G$. A functor $G:[1]\rightarrow I$ is just a 1-simplex of the nerve of the category $I$, that is, a system $(Gi, G_{i,j}:G j \rightarrow G i)$, where $Gi$, $i=0,1$, are objects of $I$ and $G_{i,j}$, $0\leq i\leq j\leq 1$, is a morphism in $I$ with $G_{i,i}=1_{Gi}$, $i=0,1$. A 2-cocycle of [1] with coefficients is $\m G$ consists of objects $Y_{i,j}\in\text{Ob}\m_{Gj}$, $0\leq i\leq j\leq 1$, with $Y_{i,i}=\mathrm{I}$ the unit object of $\m_{Gi}$. Thus an 1-simplex consist of a system of data $\left\{(G i, G_{i,j}: G j\rightarrow G i, Y_{i,j})\right\}$, $0\leq i\leq j\leq 1$, as above and therefore
$$(Ner_{I}\m)_1\cong   (\Delta^u \int_{I}\Omega^{^{-1}}\hspace{-5pt}\m)_1\,.$$

\noindent - A 2-simplex is given by a functor $G: [2]\rightarrow I$ and a 2-cocycle of $[2]$ with coefficients in $\m G$. A functor $G:[2]\rightarrow I$ is just a 2-simplex of the nerve of the category $I$, that is, a system
  $(Gi, G_{i,j}:G j \rightarrow G i)$, where $Gi$, $i=0,1,2$, are objects of $I$ and $G_{i,j}$, $0\leq i\leq j\leq 2$, is a morphism in $I$ such that the equation $G_{i,j} G_{j,k}= G_{i,k}$ holds for $0\leq i\leq j\leq k\leq 2$, with $G_{i,i}=1_{Gi}$, $i=0,1,2$.
 A 2-cocycle of $[2]$ whit coefficients in $ \m G$ is a system
$(Y_{i,j},G_{i,j,k}: G_{j,k}^*Y_{i,j}\otimes Y_{j,k}\rightarrow Y_{i,k})$ where $Y_{i,j}\in \text{Ob}\m_{Gj}$, $0\leq i\leq j\leq 2$, with $Y_{i,i}=\mathrm{I}$ the unit object of $\m_{Gi}$, and where $G_{i,j,k}$, $0\leq i\leq j\leq k\leq 2$, are morphisms in $\m_{G k}$ such that, apart from $G_{0,1,2}:G_{1,2}^*Y_{0,1}\otimes Y_{1,2}\rightarrow Y_{0,2}$,  $G_{i,i,j}$ and $G_{i,j,j}$, for any $0\leq i<j\leq 2$, are given by  canonical constraints.
 All together, we have that
$$(Ner_{I}\m)_2\cong   (\Delta^u \int_{I}\Omega^{^{-1}}\hspace{-5pt}\m)_2\,.$$

Finally,

\noindent - if $n\geq 3$, a $n$-simplex of $Ner_I\m$\ is given, on the one hand, by
 a functor $G: [n]\rightarrow I$, that is, a system of data $(Gi, G_{i,j}: Gj \rightarrow G i)_{0 \leq i<j\leq n}$,
where  $G i\in\text{Ob}I$  and $G_{i,j}$ are morphisms in $I$ such that $G_{i,j} G_{j,k}= G_{i,k}$ for $0\leq i< j< k\leq n$ and $G_{i,i}= 1_{Gi}$ for $0\leq i<j<k\leq n$. On the other hand,
 a 2-cocycle of $[n]$ with coefficients in $\m G$ is given by the system of data $(Y_{i,j},G_{i,j,k}: G_{j,k}^*Y_{i,j}\otimes Y_{j,k}\rightarrow Y_{i,k})_{0\leq i<j<k \leq n}$  where $Y_{i,j}\in \text{Ob}\m_{Gj}$, with $Y_{i,i}=\mathrm{I}$, the unit object of $\m_{Gi}$, for each $0\leq i\leq n$, and $G_{i,j,k}$\ are morphisms in $\m_{G k}$ such that, for any $0\leq i<j\leq n$, $G_{i,i,j}$ and $G_{i,j,j}$ are given by the left and right constraints of  the monoidal category  $\m_{G j}$ for the object $Y_{i,j}$. Moreover, for all $0 \leq i\leq j\leq k\leq \ell\leq n $, the coherence condition of (\ref{coh}) must hold, that is, the following diagram  must be commutative:

 $$\xymatrix@C=47pt{
G_{k,\ell}^*G_{j,k}^* Y_{i,j}\otimes G_{k,\ell}^*Y_{j,k}\otimes Y_{k,\ell}\ar[d]_{\text{can}}\ar[rr]^{1\otimes G_{j,k,\ell}}& \empty & G_{j,\ell}^* Y_{i,j}\otimes Y_{j,\ell} \ar[d]^{G_{i,j,\ell}}\\
G_{k,\ell}^*(G_{j,k}^*Y_{i,j}\otimes Y_{j,k})\otimes Y_{k,\ell} \ar[r]^-{G^{*}_{k,\ell}(G_{i,j,k})\otimes 1} & G_{k,\ell}^*Y_{i,k}\otimes Y_{k,\ell} \ar[r]^{G_{i,k,\ell}} & X_{i,\ell}\,.
}$$
 Then, for all $n\geq 0$, $$(Ner_{I}\m)_n\cong   (\Delta^u \int_{I}\Omega^{^{-1}}\hspace{-5pt}\m)_n$$
and so we have the announced isomorphism.

 \end{proof}

Since the weak homotopy equivalence (\ref{whe}) remains valid taking unitary geometric nerves instead of geometric nerves, the above proposition gives us:

\begin{theorem}\label{teomon} For any diagram of monoidal categories  $\mathcal{M}: I^{op}\rightarrow \mcat$ there is a natural weak homotopy equivalence
$$\label{result}\eta: hocolim_{I}\Delta^{u} \Omega^{^{-1}}\hspace{-5pt}\m\rightarrow \Delta^{u} \int_{I}\Omega^{^{-1}}\hspace{-5pt}\m\cong \mbox{Ner}_I\m  $$
where $\Delta^u \Omega^{^{-1}}\hspace{-5pt}\m: I^{op}\rightarrow \sset$\ is the diagram of simplicial sets obtained by composition of $\m$ with the unitary geometric nerve functor $\Delta^u: \bicat \rightarrow \sset$.
\end{theorem}

 In this way, the geometric realization of $\mbox{Ner}_I\m$ can be thought of as the homotopy colimit of the classifying spaces of the monoidal categories $\m_i $ \ given by the initial data.

\section{Homotopy colimit theorem for braided monoidal categories}
    In this section we will give an extension of Theorem \ref{teomon} to braided monoidal categories.
Recall that a \textit{braided monoidal category}, $ \mathcal{M}=( \mathcal{M}, \otimes, \boldsymbol{c} ) $,  consists of a monoidal category $\mathcal{M}=(\mathcal{M}, \otimes,a,\mathrm{I},l,r)$ together with a \textit{braiding}, that is, a family of natural isomorphisms $\boldsymbol{c}=\boldsymbol{c}_{X,Y}: X \otimes Y \rightarrow Y\otimes X$, $X,Y\in \text{Ob}\m$, satisfying suitable coherence conditions \cite{Joyal1993}.
The category of  braided monoidal categories, and braided monoidal functors between them, will be denoted by $\bmcat$.\\
Any braided monoidal category $ \mathcal{M}=( \mathcal{M}, \otimes, \boldsymbol{c} ) $ defines \cite{Berger1999, Kapranov1994}, by the categorical delooping principle, a one-object (say $\ast$), one-arrow tricategory $\Omega^{^{-2}}\hspace{-5pt}\m$ where the objects of $\m$ are the 2-cells and the morphisms are the 3-cells (thus $\Omega^{^{-2}}\hspace{-5pt}\m(\ast,\ast)=\Omega^{^{-1}}\hspace{-5pt}\m$ and  the braiding provides the interchange 3-cell between the two different composites of 2-cells).

If $\b:I^{op}\rightarrow \bmcat$ is a diagram of braided monoidal categories, the above double delooping construction determines a diagram of tricategories. A hypothetical Grothendieck construction for such a diagram, together with the consideration of suitable nerves, should allow the direct extension of Thomason's homotopy colimit theorem to braided monoidal categories. However, paralleling the monoidal case, we can avoid that Grothendieck construction to measure the homotopy colimit of the diagram, of simplicial sets, obtained from composing the diagram $\b$ with the geometric nerve of braided monoidal categories \cite[Definition 6.7]{CarCG-2010}. This is carried out by using the notion of nerve of a pseudofunctor of braided monoidal categories introduced in \cite{Cegarra2007a}. In the particular case of considering a diagram $\b:I^{op}\rightarrow \bmcat$, we shall prove  that the homotopy type of  the diagram can be represented by its nerve, $Ner_{I}\b$, a simplicial set that,  below, we recall.

A \textit{3-cocycle} of $I$\ with coefficients in $\b$\ is a system of data $(Y,f)$\ consisting of:\\
\noindent- For each two composable arrow in $I$, $k\xrightarrow{\tau} j \xrightarrow{\sigma} i$, an object $Y_{\sigma,\tau}$ of $\b_k$.\\
\noindent- For each three composable arrows in $I$, $\ell\xrightarrow{\gamma}k\xrightarrow{\tau} j \xrightarrow{\sigma} i$,
 a morphism of $\b_{\ell}$ $$f_{\sigma,\tau,\gamma}: \gamma^* Y_{\sigma,\tau}\otimes Y_{\sigma\tau,\gamma} \rightarrow Y_{\tau, \gamma} \otimes  Y_{\sigma,\tau\gamma}$$ such that,
 for any four composable arrows in $I$, $m\xrightarrow{\delta}\ell\xrightarrow{\gamma}k\xrightarrow{\tau} j \xrightarrow{\sigma} i$,
  the following diagram in $\b_{m}$ (in which we have omitted the associativity constraints) is commutative
$$\xymatrix{
\delta^*(\gamma^* Y_{\sigma,\tau}\otimes Y_{\sigma\tau,\gamma})\otimes Y_{\sigma\tau\gamma,\delta,}\ar[d]_{\delta^*f_{\sigma,\tau,\gamma}\otimes 1} \ar[r]^{\text{can}} &\delta^*\gamma^*Y_{\sigma,\tau}\otimes \delta^*Y_{\sigma\tau,\gamma}\otimes Y_{\sigma\tau\gamma,\delta}\ar[d]^{1\otimes f_{\sigma\tau,\gamma,\delta}} \\
\delta^* (Y_{\tau,\gamma}\otimes Y_{\sigma, \tau\gamma})\otimes Y_{\sigma\tau\gamma,\delta}\ar[d]_{\text{can}} &\delta^*\gamma^*Y_{\sigma,\tau}\otimes Y_{\gamma,\delta}\otimes Y_{\sigma\tau,\gamma\delta}\ar[d]^{\boldsymbol{c}\otimes 1} \\
\delta^*Y_{\tau,\gamma}\otimes \delta^* Y_{\sigma,\tau\gamma}\otimes Y_{\sigma\tau\gamma,\delta}\ar[d]_{1\otimes f_{\sigma,\tau\gamma,\delta}} & Y_{\gamma,\delta}\otimes \delta^*\gamma^*Y_{\sigma,\tau}\otimes Y_{\sigma\tau,\gamma\delta}\ar[d]^{1\otimes f_{\sigma,\tau,\gamma\delta}}\\
\delta^*Y_{\tau,\gamma}\otimes Y_{\tau\gamma,\delta}\otimes Y_{\sigma, \tau\gamma\delta}\ar[r]_{f_{\tau,\gamma,\delta}\otimes 1} & Y_{\gamma,\delta}\otimes Y_{\tau,\gamma\delta}\otimes Y_{\sigma,\tau\gamma\delta}}$$ \\
and, moreover,  $Y_{1,\sigma}=\mathrm{I}=Y_{\sigma,1}$, $f_{\sigma,\tau,1}= \boldsymbol{c}_{_{Y_{\sigma,\tau},\mathrm{I}}}:Y_{\sigma,\tau}\otimes \mathrm{I}\rightarrow \mathrm{I}\otimes Y_{\sigma,\tau}$, $f_{\sigma, 1,\gamma}:\gamma^*\mathrm{I}\otimes Y_{\sigma,\gamma}\rightarrow \mathrm{I}\otimes Y_{\sigma,\gamma}$ is the composite of $ 1_{\mathrm{I} \otimes Y_{\sigma,\gamma}}$ with the unit constraint of the monoidal functor $\gamma^*$ and $f_{1,\tau,\gamma}:\gamma^*\mathrm{I}\otimes Y_{\tau,\gamma}\rightarrow Y_{\tau,\gamma}\otimes \mathrm{I}$ is the composite of $\boldsymbol{c}_{_{\mathrm{I},Y_{\tau,\gamma}}}$ with the unit constraint of the monoidal functor $\gamma^*$.

Then, the nerve $Ner_I\b$ of the diagram is  defined as the simplicial set, \cite{Cegarra2007a},
\begin{equation}\label{nerdiag}
Ner_{I}\b:\left[ n \right] \ \mapsto \underset{G: [n]\rightarrow I}{\coprod}  Z^3([n],\b G)\end{equation}
where $G: [n]\rightarrow I$\ is any functor and $Z^3([n],\b G)$\ is the set of 3-cocycles of $[n]$\ with coefficients in the composite functor $[n]\xrightarrow{G} I\xrightarrow{\b} \bmcat$. Thus, an n-simplex can be described as  a system of data
$$\mathbb{G}=\left\{ G_{i,j,k},\ G_{i,j,k,\ell}\right\}_{0\leq i\leq j\leq k\leq \ell\leq n} $$
where:

 \noindent - $G_{i,j,k}\in \text{Ob}\b_{G k}$ for $0\leq i\leq j \leq k\leq n$, with $G_{i,i,j}=G_{i,j,j}=\mathrm{I} $, the unit object of the braided monoidal category $\b_{Gj}$.

\noindent - $G_{i,j,k,\ell}:G_{k,\ell}^{*}G_{i,j,k}\otimes G_{i,k,\ell}\rightarrow G_{j,k,\ell}\otimes G_{i,j,\ell}$, for $0\leq i\leq j\leq k \leq \ell \leq n$, is a morphism in $\b_{G \ell}$, where $G_{k,\ell}^*:\b_{Gk}\rightarrow \b_{G\ell}$\ is the braided monoidal functor associated to the morphism $\ell\rightarrow k$\ of $[n]$, with
$G_{i,i,j,k}$ the composite of $\boldsymbol{c}_{\mathrm{I},G_{i,j,k}}$ with the unit constraint of $G_{j,k}^*$, $G_{i,j,j,k}$ the composite of $1_{\mathrm{I}\otimes G_{i,j,k}}$ with the unit constraint of $G_{j,k}^*$ and $G_{i,j,k,k}=\boldsymbol{c}_{G_{i,j,k},\mathrm{I}}$.

Moreover, for all $0\leq i\leq j\leq k\leq \ell\leq m \leq n$, the following diagram in $\b_{G m}$ must be commutative:\\
\begin{equation}\label{diag1}\xymatrix{
G_{\ell,m}^*(G_{k,\ell}^* G_{i,j,k}\otimes G_{i,k,\ell})\otimes G_{i,\ell,m}\ar[d]_{G_{\ell,m}^*G_{i,j,k,\ell}\otimes 1} \ar[r]^-{\text{can}} &  G_{\ell,m}^*G_{k,\ell}^*G_{i,j,k}\otimes G_{\ell,m}^*G_{i,k,\ell}\otimes G_{i,\ell,m}\ar[d]^{1\otimes G_{i,k,\ell,m}}\\
G_{\ell,m}^* (G_{j,k,\ell}\otimes G_{i,j,\ell})\otimes G_{i,\ell,m}\ar[d]_{\text{can}} & G_{\ell,m}^*G_{k,\ell}^*G_{i,j,k}\otimes G_{k,\ell,m}\otimes G_{i,k,m}\ar[d]^{\boldsymbol{c}\otimes 1}\\
G_{\ell,m}^*G_{j,k,\ell}\otimes G_{\ell,m}^* G_{i,j,\ell}\otimes G_{i,\ell,m}\ar[d]_{1\otimes G_{i,j,\ell,m}} &  G_{k,\ell,m}\otimes G_{k,m}^*G_{i,j,k}\otimes G_{i,k,m}\ar[d]^{1\otimes G_{i,j,k,m}} \\
G_{\ell,m}^*G_{j,k,\ell}\otimes G_{j,\ell,m}\otimes G_{i,j,m}\ar[r]_{G_{j,k,\ell,m}\otimes 1} & G_{k,\ell,m}\otimes G_{j,k,m}\otimes G_{i,j,m}\,.}
 \end{equation}

Particularly, if the diagram $\b:I^{op}\rightarrow \bmcat$ is  constant a  braided monoidal category  $ \mathcal{M}=( \mathcal{M}, \otimes, \boldsymbol{c})$, the notion of 3-cocycle of $I$ with coefficients in $\b$ is just that of 3-cocycle of $I$ in $\m$ (\cite[Definition 6.6]{CarCG-2010}) and, in this case, $Ner_{I}\b$ is the geometric nerve, $Ner\m=Z^3(\m,\otimes,\boldsymbol{c})$, of the braided monoidal category $\m$ (\cite[Definition 6.7]{CarCG-2010})
that is, the simplicial set:
$$ \begin{aligned}
Ner\mathcal{M}:\ & \Delta^{op} \rightarrow Set\\
& \left[ n \right]\ \mapsto Z^{3}([n],\mathcal{M})
\end{aligned}$$
where $Z^{3}([n],\mathcal{M})$\ is the set of 3-cocycles of $[n]$\ in $\mathcal{M}$. This defines the \textit{geometric nerve functor of braided monoidal categories}, $Ner: \bmcat \rightarrow \sset $, which associates with each braided monoidal category $\mathcal{M}$ its  geometric nerve $Ner \mathcal{M}$.

For any diagram $\b: I^{op}\rightarrow \bmcat$ we can consider the bisimplicial set
$$S= \coprod_{G \in Ner I}Ner \b_{G 0}= \coprod_{G: \left[ q \right]\rightarrow I}Z^{3}([p],\b_{G 0})$$
 whose$(p,q)$-simplices  are pairs $(\mathbb{G},G)$\ where $G : \left[ q \right] \rightarrow I$\ is a functor and $\mathbb{G}: \left[ p \right] \rightarrow \b_{G 0}$\ is a 3-cocycle of $[p]$\ in $\b_{G 0}$.
If $\alpha: \left[ p' \right]\rightarrow \left[p\right]$\ and $\beta: \left[ q' \right]\rightarrow \left[ q \right] $\  are maps in the simplicial category, then the respective horizontal and vertical induced maps are defined by:
$$
\alpha^{*h}(\mathbb{G},G)=(\mathbb{G}\alpha,G)\,\,;\,\,
\alpha^{*v}(\mathbb{G},G)=(G^{*}_{0,\beta 0}\mathbb{G}, G \beta)
$$
where $G\beta: \left[ q'\right]\xrightarrow{\beta}\left[ q\right]\xrightarrow{G}\mathcal{I}$, $\mathbb{G}\alpha : \left[ p' \right]\xrightarrow{\alpha}\left[ p \right]\xrightarrow {\mathbb{G}}\b_{G 0}$ is a 3-cocycle of $\left[ p'\right]$\ in $\b_{G 0}$\ and $G_{0,\beta 0}^{*}\mathbb{G}: \left[ p \right]\rightarrow \b_{G\beta 0} $\ is the 3-cocycle of $\left[ p \right]$\ in $\b_{G\beta 0}$\ obtained by  composing  $\mathbb{G}$\ with the braided monoidal functor $G^{*}_{0,\beta 0}$\ associated to the morphism $G\beta 0 \rightarrow G 0$\ of $I$.
In particular, the horizontal and vertical faces of $S$ are given by
 $$d_i^{h} (\mathbb{G},G)=(\mathbb{G} d^i , G)\,,\, 0\leq i\leq p\,,$$ and $$d_j^{v} (\mathbb{G},G)=(\mathbb{G}, G d^j)\,, 1\leq j \leq q\,, \,\,\text{while}\,\,\, d_0^v(\mathbb{G},G)=(G_{0,1}^*\mathbb{G},Gd^0)\,.
 $$

The diagonal of this bisimplicial set is just the homotopy colimit, $hocolim_{I}Ner\b$, of the diagram of simplicial sets
 $I^{op}\xrightarrow{\b} \bmcat \xrightarrow{Ner} \sset$.
Thus,  $hocolim_I Ner\mathcal{C}=\text{diag}\,S$\ is the simplicial set whose n-simplices are pairs $(\mathbb{G},G)$\ where $G : \left[ n \right] \rightarrow I$\ is a functor and $\mathbb{G}: \left[ n \right] \rightarrow \b_{G 0}$\ is a 3-cocycle of $[n]$\ in $\b_{G 0}$.\\

Now we are ready to prove our main result:

\begin{theorem}\label{teor} (\emph{Homotopy Colimit Theorem for braided monoidal categories})
For any diagram of braided monoidal categories $\b: I^{op}\rightarrow \bmcat$\ there exists a natural weak homotopy equivalence of simplicials sets
\begin{eqnarray} \eta : hocolim_{I} Ner \b\rightarrow Ner_{I}\b \end{eqnarray}
where $Ner \b: I^{op}\rightarrow \sset$\ is the diagram of simplicial sets, obtained by the composition of the diagram $\b$\ with the geometric nerve functor of braided monoidal categories $Ner : \bmcat\rightarrow \sset$, and $Ner_{I}\b$ is the nerve (\ref{nerdiag}) of the diagram.
\end{theorem}
\begin{proof}
Let $(\mathbb{G},G)$ be an $n$-simplex of $hocolim_{I}Ner\b$. Then $G: [n]\rightarrow I$ is a functor and $\mathbb{G}: [n ]\rightarrow \b_{G 0}$
is a 3-cocycle of $[n]$\ in $\b_{G 0}$, that is, a system of data
$$\mathbb{G}=\{G_{i,j,k},G_{i,j,k,\ell}:G_{i,j,k}\otimes G_{i,k,\ell}\rightarrow G_{j,k,\ell}\otimes G_{i,j,\ell}\}_{0\leq i\leq j\leq k\leq \ell\leq n}$$
where $G_{i,j,k}$\ are objects and $G_{i,j,k,\ell}$\ are morphisms of $\b_{G 0} $ such that:

 \noindent -  $G_{i,i,j}=F_{i,j,j}=\mathrm{I} $, the unit object of $\b_{G0}$;

\noindent - $G_{i,i,j,k}$ is the composite of $\boldsymbol{c}_{\mathrm{I},G_{i,j,k}}$ with the unit constraint of $G_{j,k}^*$, $G_{i,j,j,k}$ is the composite of $1_{\mathrm{I}\otimes G_{i,j,k}}$ with the unit constraint of $G_{j,k}^*$ and $G_{i,j,k,k}=\boldsymbol{c}_{G_{i,j,k},\mathrm{I}}$;

\noindent and, moreover, for all $0\leq i\leq j\leq k\leq \ell\leq m \leq n$, the following diagram in $\b_{G 0}$  is commutative:\\
\begin{equation}\label{diag2}\xymatrix@C=58pt{
( G_{i,j,k}\otimes G_{i,k,\ell})\otimes G_{i,\ell,m}\ar[d]_{G_{i,j,k,\ell}\otimes 1} \ar[r]^-{\text{can}} &  G_{i,j,k}\otimes (G_{i,k,\ell}\otimes G_{i,\ell,m})\ar[d]^{1\otimes G_{i,k,\ell,m}}\\
 (G_{j,k,\ell}\otimes G_{i,j,\ell})\otimes G_{i,\ell,m}\ar[d]_{\text{can}} & G_{i,j,k}\otimes (G_{k,\ell,m}\otimes G_{i,k,m})\ar[d]^{\text{can}(\boldsymbol{c}\otimes 1)\text{can}}\\
G_{j,k,\ell}\otimes ( G_{i,j,\ell}\otimes G_{i,\ell,m})\ar[d]_{1\otimes G_{i,j,\ell,m}} &  G_{k,\ell,m}\otimes (G_{i,j,k}\otimes G_{i,k,m})\ar[d]^{\text{can}(1\otimes G_{i,j,k,m})} \\
G_{j,k,\ell}\otimes (G_{j,\ell,m}\otimes G_{i,j,m})\ar[r]_{(G_{j,k,\ell,m}\otimes 1)\text{can}} & (G_{k,\ell,m}\otimes G_{j,k,m})\otimes G_{i,j,m}\,.}
 \end{equation}

Then we  define a map
$$\eta : (hocolim_{I}Ner\b)_n\rightarrow (Ner_{I}\b)_n\,\,,\,\,\,  (\mathbb{G},G)\mapsto (\mathbb{G}',G)$$
where $$\mathbb{G}'={\left\{ G'_{i,j,k},\ G'_{i,j,k,\ell}:G_{k,\ell}^*G'_{i,j,k}\otimes G'_{i,k,\ell}\rightarrow G'_{j,k,\ell}\otimes G'_{i,j,\ell}\right\}}_{0\leq i\leq j\leq k\leq \ell\leq n}$$\ is the 3-cocycle of $[n]$\ with coefficient in   $\b G$, defined as follows:

\noindent - The objects of $\b_{Gk}$, $G'_{i,j,k}= G_{0,k}^* G_{i,j,k}$, for $0\leq i < j < k \leq n$, and $G'_{i,i,j}=G'_{i,j,j}=\mathrm{I}$, $0\leq i\leq j\leq n$, where $\mathrm{I}$ denotes  the unit object of $\b_{Gj}$.

\noindent - The morphisms $G'_{i,j,k,\ell}$ in $\b_{G\ell}$, $0\leq i\leq j\leq k\leq \ell\leq n$, given as the dotted arrow in the following diagram:

$$\xymatrix@C=50pt{G_{k,\ell}^*G'_{i,j,k}\otimes G'_{i,k,\ell}\ar@{.>}[r]^{G'_{i,j,k,\ell}}\ar[d]_{\text{can}}&G'_{j,k,\ell}\otimes G'_{i,j,\ell} \\
G_{0,\ell}^*(G_{i,j,k}\otimes G_{i,k,\ell})\ar[r]^{G_{0,k}^*G_{i,j,k,\ell}}&G_{0,\ell}^*(G_{j,k,\ell}\otimes G_{i,j,\ell})\,.\ar[u]_{\text{can}}}$$

It is straightforward to check that, in this way,  $(\mathbb{G}', G)$ is actually an $n$-simplex of $Ner_I\b$. For instance, the commutativity of diagram (\ref{diag1}) for $\mathbb{G}'$ is deduced from the commutativity of (\ref{diag2}) for $\mathbb{G}$.

Since $hocolim_{I}Ner\b=\text{diag} S$ we have, according to (\ref{Phi}), a natural weak homotopy equivalence
$$\Phi: hocolim_{I}Ner\b \rightarrow \overline{W}S$$
and we will prove that $\eta$ is also a weak homotopy equivalence by showing a simplicial isomorphism $$\Psi: \overline{W}S \cong Ner_{I}\b$$
making the following diagram of simplicial sets commutative:
$$\xymatrix{
hocolim_{I}Ner \b\ar[dr]_{\Phi}^{\simeq} \ar[rr]^{\eta}&\empty & Ner_{I}\b\\
\empty & \overline{W}S\,. \ar[ur]_{\Psi}^{\cong}& \empty }$$

According to the general description of the simplices of  $\overline{W}S$ recalled in Preliminaries,
a $p$-simplex de $\overline{W}S$, in our case, can be described as a list of pairs:
$$\chi = ((\mathbb{G}^{(0},G^{(p}),\cdots,(\mathbb{G}^{(m},G^{(p-m}),\cdots,(\mathbb{G}^{(p},G^{(0})) $$
where each $G^{(p-m}:[p-m]\rightarrow I$ is a  functor and each $\mathbb{G}^{(m}:[m]\rightarrow \b_{G^{(p-m}0}$\ is a 3-cocycle of $[m]$\ in $\b_{G^{(p-m}0}$, such that the following equalities
$$G^{(p-m}d^0 = G^{(p-m-1},\ \ \  G_{0,1}^{(p-m\ *}\mathbb{G}^{(m}=\mathbb{G}^{(m+1}d^{m+1}$$
hold, for all $0 \leq m < p$.

Writing $G^{(p}:[p]\rightarrow I$\ simply as $G :[p]\rightarrow I$, an iterated use of the above equalities proves that
$$G^{(p-m} = G d^{0}\overset{(m}\cdots d^0:[p-m]\rightarrow I,\ \ \ 0\leq m \leq p\,,\,\, \text{and}$$
$$\mathbb{G}^{(m+1}d^{m+1}\cdots d^{k+1}= G^{*}_{k,m+1}\mathbb{G}^{(k}:[k]\rightarrow \b_{G(m+1)}\ \, \ 0\leq k\leq m<p\,.$$
Since each 3-cocycle $\mathbb{G}^{(m}$ is a system of data $\{G^{(m}_{i,j,k}, G^{(m}_{i,j,k,\ell}\}$, these latter equations mean that:
$$G^{(\ell}_{i,j,k}=G^{*}_{k,\ell}G^{(k}_{i,j,k}\,\,,\, i \leq j \leq k \leq \ell\,\,; \,\,
G^{(m}_{i,j,k,\ell}= G^*_{\ell,m}G^{(\ell}_{i,j,k,\ell}\,\,,\,\, i\leq j \leq k \leq \ell \leq m\,.$$

 Thus, a $p$-simplex $\chi$\ of $\overline{W}S$\ is uniquely determined by a functor $G:[p]\rightarrow I$, the  objects $G^{(k}_{i,j,k}$\ of $\b_{G k}$ and the morphisms $G^{(\ell}_{i,j,k,\ell}$\ of $\b_{G \ell}$, for any $0\leq i\leq j\leq k \leq \ell\leq p$. Then we observe that there is a 3-cocycle
$\mathbb{G}'=\{G'_{i,j,k},G'_{i,j,k,\ell}\}$\ of $[p]$\ with coefficients in $\b G$\ defined by $G'_{i,j,k}=G^{(k}_{i,j,k}$, for any $0\leq i<j<k\leq p$, and $G'_{i,j,k,\ell}=G^{(\ell}_{i,j,k,\ell}$, for any $0\leq i<j<k<\ell\leq p$, and therefore
   a $p$-simplex $\chi$\ of $\overline{W}S$\ defines a $p$-simplex $(\mathbb{G}',G)$\ of $Ner_{I}\b$, which itself uniquely determines $\chi$. In this way, we get an injective simplicial map
$$\begin{array}{c c}
\Psi: \overline{W}S \longrightarrow Ner_{I}\b\\
((\mathbb{G}^{(0},G^{(p}),\cdots,(\mathbb{G}^{(p},G^{(0}))\mapsto (\mathbb{G}',G)=(\{G^{(k}_{i,j,k},\ G^{(l}_{i,j,k,l}\},G^{(p}\ )\,.
\end{array}$$
But also $\Psi$ is surjective. In fact, let $(\mathbb{G}',G)$ any $p$-simplex\ of $Ner_{I}\b$, that is, let $G:[p]\rightarrow I$ be a functor and let $\mathbb{G}'=\left\{G'_{ijk},G'_{ijk\ell}\right\}$ be 3-cocycle of $[p]$\ with coefficients in $\b G$. Then we can consider the  $p$-simplex $\chi=(\mathbb{G}^{(m},G^{(p-m})$\ of $\overline{W}S$\ where, for each $0\leq m\leq p$, $G^{(p-m}:[p-m]\rightarrow I$\ is the composite $[p-m]\xrightarrow{(d^0 )^m}[p] \xrightarrow{G} I$\ and the 3-cocycle $\mathbb{G}^{(m}$\ of $[m]$\ in $\b_{G^{(p-m}0}$\ is defined as follows:

\noindent - The objects $G^{(m}_{i,j,k}=G_{k,m}^* G'_{i,j,k}$;

\noindent - The morphisms $G^{(m}_{i,j,k,\ell}:G^{(m}_{i,j,k}\otimes G^{(m}_{i,k,\ell}\rightarrow G^{(m}_{j,k,\ell}\otimes G^{(m}_{i,j,\ell}$\  are given as the dotted arrow in the following diagram

$$\xymatrix@C=50pt{G_{k,m}^* G'_{i,j,k}\otimes G_{\ell,m}^* G'_{i,k,\ell}\ar[d]_{\text{can}}\ar@{.>}[r]^{G^{(m}_{i,j,k,\ell}}&G^*_{\ell,m}G'_{j,k,\ell}\otimes G_{\ell,m}G'_{i,j,\ell}\\G_{\ell,m}^*(G_{k,\ell}^*G'_{i,j,k}\otimes G'_{i,k,\ell})\ar[r]_{G^*_{\ell,m}G'{i,j,k,\ell}}&G^*_{\ell,m} (G'_{j,k,\ell}\otimes G'_{i,j,\ell})\ar[u]_{\text{can}}\,.}$$

 It is easy to check that $\Psi(\chi)=(\mathbb{G}',G)$, whence we conclude that the simplicial map $\Psi$\ is surjective and, therefore, it is an isomorphism.

Finally, since $\eta= \Psi\Phi$,   $\Psi$ is an isomorphism and $\Phi$\ is a natural weak homotopy equivalence, we have that $\eta$\ is a natural weak homotopy equivalence as claimed.
\end{proof}
\par

This theorem allows to think of the geometric realization of $Ner_I\b$ as the homotopy colimit of the classifying spaces of the braided monoidal categories $\b_i$ given by the initial data of the diagram  $\b:\mathcal{I}^{op}\rightarrow \bmcat$.

\bibliography{biblioCatHo}{}
\bibliographystyle{plain}

\end{document}